\documentclass[12pt, a4paper]{amsart}
\usepackage[utf8]{inputenc}
\usepackage{amsfonts,amssymb,amsxtra,amsthm,amsmath,amscd,mathrsfs,cite}
\usepackage[all]{xy}%mj-eucal,eucal 可修改 LaTeX 的数学字体命令 \mathcal。当加载该宏包后，使用 \mathcal 命令，调出的是欧拉书写体，而不是通常的计算机现代书写体。
\usepackage{tikz}
\usepackage{verbatim}
\usepackage[normalem]{ulem}
\usepackage{soul}
\usepackage{color}

\setstcolor{red}

\makeatletter
\@namedef{subjclassname@2010}{%
	\textup{2010} Mathematics Subject Classification}
\makeatother

\def\geq{\geqslant}

\theoremstyle{plain}
\newtheorem{theorem}{Theorem}[section]
\theoremstyle{definition}

\newtheorem{Conjecture}[theorem]{Conjecture}

\newtheorem{Lemma}[theorem]{Lemma}

\theoremstyle{remark}

\numberwithin{equation}{section}

\addtocounter{footnote}{1}

\begin{document}

\title[On the Double Lambert Series Conjecture of Andrews--Dixit--Schultz--Yee]{On the Double Lambert Series Conjecture of Andrews--Dixit--Schultz--Yee}

\author{Qianwen Fang}
\address{School of Economics, Shandong University, Jinan, Shandong, People's Republic of China}
\email{qwfang1998@163.com}
\subjclass[2020]{11B65, 11P81}
\keywords{Lambert series, Double Lambert series}
\thanks{ORCID:0009-0001-7695-2599}
\begin{abstract}
Andrews, Dixit, Schultz, and Yee conjecture the parity of a double Lambert series. In 2026, Amdeberhan, Andrews, and Ballantine offer some ideas that are pointing in the right direction for the proof. In this paper, we complete the rest of their proof.
\end{abstract}
\maketitle

\section{Introduction}

In \cite{2}, Andrews, Dixit, Schultz, and Yee proposed the following conjecture regarding the parity of a double Lambert series.

\begin{Conjecture}
\label{c1}
The following is an odd function of $q$, $|q|<1$:
\begin{align*}
Y(q):=\sum_{m,n\geq 1} \frac{(-q)^{2mn+m}}{(1+q^m)(1-q^{2m-1})}.
\end{align*}
\end{Conjecture}

As shown in \cite{3}, $Y(q)$ can be expressed as
\begin{align}
\label{2}
Y(q)=-\sum_{k\geq 2}\frac{q^k}{1+q^{2k-1}}\sum_{n=1}^{k-1}\frac{q^n}{1+q^n},
\end{align}
However, the authors were unable to complete the proof of Conjecture \ref{c1} using this representation. The difficulty in the final step may arise from the formulation in \eqref{2}; instead, we consider the following form:
\begin{align}
\label{1}
Y(q)&=\sum_{\substack{m,n\geq 1 \\ k,l\geq 0}}(-1)^{m+k} q^{2mn+nk+2ml+m-l}=\sum_{\substack{m \geq 1 \\ k \geq 0}}\frac{(-1)^{m+k}q^{3m+k}}{(1-q^{2m-1})(1-q^{2m+k})}\\
\nonumber
&=\sum_{m \geq 1}\frac{(-1)^mq^m}{1-q^{2m-1}}\sum_{k=2m}^\infty \frac{(-1)^kq^k}{1-q^k}.
\end{align}
Building on the approach in \cite[Theorem 5.7]{3}, we provide the remaining steps to complete the proof.

After this paper was written, we learned from George Andrews that Cui and Tang had also independently proven Conjecture \ref{c1} using a similar approach.

Throughout this paper, we adopt the standard notation for the $q$-Pochhammer symbol
\[
(a; q)_\infty := \prod_{n=0}^{\infty} (1 - aq^n).
\]
Furthermore, for the product of multiple $q$-Pochhammer symbols, we use the following condensed notation:
\[
(a_1, a_2, \dots, a_k; q)_\infty := (a_1; q)_\infty (a_2; q)_\infty \dots (a_k; q)_\infty.
\]

\section{Proof of Conjecture \ref{c1}}

We define the following auxiliary functions
\begin{align*}
Z(q):=\sum_{m \geq 1}\frac{(-1)^mq^m}{1-q^{2m-1}}\sum_{k=1}^{2m-1} \frac{(-1)^kq^k}{1-q^k},
\end{align*}
\begin{align*}
A(q):=\sum_{i \geq 0}\sum_{j=i+1}^\infty \frac{q^{j+1}}{(1+q^{2i+1})(1+q^{2j+1})},
\end{align*}
\begin{align*}
B(q):=\sum_{i \geq 0}\sum_{j=i+1}^\infty \frac{q^{i+2j+2}}{(1+q^{2i+1})(1+q^{2j+1})},
\end{align*}
\begin{align*}
B_1(q):=\sum_{i \geq 0}\sum_{j=0}^{i} \frac{q^{i+2j+2}}{(1+q^{2i+1})(1+q^{2j+1})},
\end{align*}
\begin{align*}
D_1(q):=Y(q)+Z(q)=\sum_{m \geq 1}\frac{(-1)^mq^m}{1-q^{2m-1}}\sum_{k\geq 1} \frac{(-1)^kq^k}{1-q^k},
\end{align*}
and
\begin{align*}
D_2(q):=B(q)+B_1(q)&=\sum_{i \geq 0}\frac{q^i}{1+q^{2i+1}}\sum_{j\geq 0} \frac{q^{2j+2}}{1+q^{2j+1}}\\
&=\sum_{m \geq 1}\frac{(-1)^mq^m}{1-q^{2m-1}}\sum_{k\geq 1} \frac{(-1)^kq^{k}}{1-q^{2k}}.
\end{align*}
The last equality can be verified in a similar fashion to \eqref{1}. A direct calculation reveals that
\begin{align*}
Z(q)&=\sum_{\substack{k\geq 1 \\ m,i,j\geq 0}}(-1)^{m+k+1}q^{3k+m+2mi+2ki+2kj-i-j-1}+\sum_{\substack{m,k\geq 1 \\ i,j\geq 0}}(-1)^{m+k}q^{3k+m+2mi+2ki+2kj-i}\\
&=A(q)+B(q).
\end{align*}
Consequently, we obtain the relation
$$
Y(q)=D_1(q)-D_2(q)-A(q)+B_1(q).
$$
Conjecture \ref{c1} then follows from the two lemmas below.

\begin{Lemma}
We have the following identity:
$$
B_1(q)=A(-q).
$$
\end{Lemma}
\begin{proof}
\begin{align*}
B_1(q)&=\sum_{i,j\geq 0} \frac{q^{i+3j+2}}{(1+q^{2i+2j+1})(1+q^{2j+1})}\\
&=\sum_{m,n,i,j\geq 0}(-1)^{m+n}q^{2im+2jm+2jn+n+i+m+3j+2}\\
&=\sum_{m,n,i,j\geq 0}(-1)^{i+j}(-q)^{2im+2jm+2jn+n+i+m+3j+2}\\
&=\sum_{m,n\geq 0} \frac{(-q)^{m+n+2}}{(1+(-q)^{2n+2m+3})(1+(-q)^{2m+1})}\\
&=A(-q).
\end{align*}
\end{proof}

\begin{Lemma}
The following relation holds:
$$
D_1(q)-D_2(q)=q\frac{(q^4;q^4)_\infty^4}{(q^2;q^2)_\infty^2}\sum_{k\geq 1}\frac{(-1)^{k-1}q^{2k}}{1-q^{2k}}.
$$
\end{Lemma}
\begin{proof}
Recall the following well-known identity from $q$-series theory (see, e.g., \cite[Entry 29]{1}):
$$
\sum_{n\in \mathbb{Z}}\frac{x^n}{1-yq^n}=\frac{(q,q,xy,q/xy;q)_\infty}{(x,q/x,y,q/y;q)_\infty},
$$
Setting appropriate parameters, we obtain
\begin{align*}
\sum_{m \geq 1}\frac{(-1)^mq^m}{1-q^{2m-1}}=\frac{1}{2}\sum_{m \in \mathbb{Z}}\frac{(-1)^mq^m}{1-q^{2m-1}}=q\frac{(q^4;q^4)_\infty^4}{(q^2;q^2)_\infty^2}.
\end{align*}
Furthermore, by observing the terms of $D_1(q)$ and $D_2(q)$, we have
\begin{align*}
\sum_{k\geq 1} \frac{(-1)^kq^k}{1-q^k}-\sum_{k\geq 1} \frac{(-1)^kq^{k}}{1-q^{2k}}=\sum_{k\geq 1}\frac{(-1)^{k-1}q^{2k}}{1-q^{2k}}.
\end{align*}
This completes the proof.
\end{proof}

\section{Future work}

The proof of Conjecture \ref{c1} would be greatly simplified if the following conjecture could be established through a more elementary method.
\begin{Conjecture}
$$
Y(q)=D_2(q)-D_1(q).
$$
\end{Conjecture}

\end{document}